\documentclass{amsart}

\usepackage{amssymb}
\usepackage{mathrsfs,mathtools}
\usepackage{enumerate}
\usepackage{esint}
\usepackage{titletoc}
\usepackage[colorlinks=true, urlcolor=red, linkcolor=red, citecolor=blue]{hyperref}

\theoremstyle{plain}
\newtheorem{theorem}{Theorem}[section]
\newtheorem{lemma}[theorem]{Lemma}

\newtheorem{proposition}[theorem]{Proposition}

\theoremstyle{definition}
\newtheorem{definition}[theorem]{Definition}

\newtheorem{example}[theorem]{Example}

\numberwithin{equation}{section}

\makeatletter
\@namedef{subjclassname@2020}{\textup{2020} Mathematics Subject Classification}
\makeatother

\title{Removable singularities for nonlocal minimal graphs}

\author{Minhyun Kim}
\address{Department of Mathematics \& Research Institute for Natural Sciences, Hanyang University, 04763 Seoul, Republic of Korea}
\email{minhyun@hanyang.ac.kr}

\thanks{M.~Kim is supported by the National Research Foundation of Korea (NRF) grant funded by the Korean government (MSIT) (RS-2023-00252297).}

\begin{document}

\begin{abstract}
We prove the removable singularity theorem for nonlocal minimal graphs. Specifically, we show that any nonlocal minimal graph in $\Omega \setminus K$, where $\Omega \subset \mathbb{R}^n$ is an open set and $K \subset \Omega$ is a compact set of $(s, 1)$-capacity zero, is indeed a nonlocal minimal graph in all of $\Omega$.
\end{abstract}

\maketitle

\section{Introduction}

It is well known that solutions of the minimal surface equation cannot have singularities occupying a set of $(n-1)$-dimensional Hausdorff measure zero. Specifically, if $u$ is a continuous weak solution of
\begin{equation}\label{eq-MSE}
\mathrm{div}\left( \frac{Du}{\sqrt{1+|Du|^2}} \right) = 0
\end{equation}
in $\Omega \setminus K$, where $\Omega \subset \mathbb{R}^n$ is an open set and $K \subset \Omega$ is a compact set of $(n-1)$-dimensional Hausdorff measure zero (or more generally, a compact set of $1$-capacity zero), then $u$ can be extended to a continuous weak solution of \eqref{eq-MSE} in $\Omega$. This result was first established by Bers~\cite{Ber51} when $K$ consists of a single point. Later, the full result was independently proved by Nitsche~\cite{Nit65} for $n=2$ and by De~Giorgi--Stampacchia~\cite{DGS65} for general $n \geq 2$.

The removable singularity theorem for the minimal surface equation has been generalized to a large class of elliptic equations. For instance, Serrin~\cite{Ser64,Ser65} proved the same result for quasilinear elliptic equations including equations of bounded prescribed mean curvature. Moreover, V\'azquez--V\'eron~\cite{VV81} extended this result further to include the capillarity equation.

The aim of this paper is to study removable singularities for nonlocal minimal graphs, which are solutions of the nonlocal minimal surface equation. To state the main results, we first provide several definitions and recall known results in the literature. The study begins with the $s$-perimeter, which was first introduced in the influential paper by Caffarelli--Roquejoffre--Savin~\cite{CRS10}. Given $s \in (0,1)$ and $n \in \mathbb{N}$, the \emph{$s$-perimeter} of a measurable set $E \subset \mathbb{R}^{n+1}$ in an open set $\mathcal{O} \subset \mathbb{R}^{n+1}$ is defined by
\begin{equation*}
\mathrm{Per}_s(E, \mathcal{O}) \coloneqq L_s(E \cap \mathcal{O}, \mathbb{R}^{n+1} \setminus E) + L_s(E \setminus \mathcal{O}, \mathcal{O} \setminus E),
\end{equation*}
where
\begin{equation*}
L_s(A, B) \coloneqq \int_A \int_B \frac{\mathrm{d}Y \,\mathrm{d}X}{|X-Y|^{n+1+s}}, \quad A, B \subset \mathbb{R}^{n+1}\text{ measurable}.
\end{equation*}
We say that a measurable set $E \subset \mathbb{R}^{n+1}$ is \emph{$s$-minimal} in an open set $\mathcal{O} \subset \mathbb{R}^{n+1}$ if $\mathrm{Per}_s(E, \mathcal{O})< \infty$ and
\begin{equation*}
\mathrm{Per}_s(E, \mathcal{O}) \leq \mathrm{Per}_s(F, \mathcal{O}) \quad\text{for every }F \subset \mathbb{R}^{n+1} \text{ s.t.~}F\setminus \mathcal{O} = E \setminus \mathcal{O}.
\end{equation*}
We also say that $E$ is \emph{locally $s$-minimal} in $\mathcal{O}$ if $E$ is $s$-minimal in every open set $\mathcal{O}' \Subset \mathcal{O}$. The boundary $\partial E$ of the $s$-minimal or locally $s$-minimal set $E$ is often called \emph{nonlocal minimal surface} in $\mathcal{O}$. In this work, we focus on the locally $s$-minimal sets, and the term ``nonlocal minimal surface'' refers to these sets.

The main result in this paper is concerned with a particular class of nonlocal minimal surfaces; \emph{nonlocal minimal graphs}. We call $\partial E$ a \emph{nonlocal minimal graph} in an open set $\Omega \subset \mathbb{R}^n$ if $E$ is locally $s$-minimal in the infinite vertical cylinder $\Omega \times \mathbb{R}$ and if $E$ can be written as the entire subgraph of a measurable function $u: \mathbb{R}^n \to \mathbb{R}$, i.e.,
\begin{equation*}
E=\mathcal{S}_u \coloneqq \{(x, t) \in \mathbb{R}^n \times \mathbb{R}: t<u(x) \}.
\end{equation*}
The study on nonlocal minimal graphs has been motivated by the work \cite{DSV16}, where Dipierro--Savin--Valdinoci proved that an $s$-minimal set $E$ in a cylinder $\mathcal{O}=\Omega \times \mathbb{R}$, which is the subgraph of a continuous function outside $\mathcal{O}$, remains a subgraph in the entire space. The existence of such a minimizing set was established by Lombardini~\cite{Lom18}. For the smoothness of nonlocal minimal graphs, we refer the reader to Caffarelli--Roquejoffre--Savin~\cite{CRS10}, Barrios--Figalli--Valdinoci~\cite{BFV14}, Figalli--Valdinoci~\cite{FV17}, and Cabr\'e--Cozzi~\cite{CC19}.

In \cite{CRS10}, it was proved that the Euler--Lagrange equation of $\mathrm{Per}_s$ in $\mathcal{O}$ is given by
\begin{equation*}
H_s[E](X)=0, \quad X \in \mathcal{O}\cap \partial E,
\end{equation*}
where
\begin{equation*}
H_s[E](X) \coloneqq \mathrm{p.v.} \int_{\mathbb{R}^{n+1}} \frac{\chi_{\mathbb{R}^{n+1} \setminus E}(Y) - \chi_E(Y)}{|X-Y|^{n+1+s}} \,\mathrm{d}Y
\end{equation*}
denotes the \emph{nonlocal mean curvature} (or \emph{$s$-mean curvature}) of $E$ at a point $X \in \partial E$. Caffarelli--Valdinoci~\cite{CV13} proved that the nonlocal mean curvature converges to the classical mean curvature as $s \nearrow 1$. Moreover, they also proved that if $E$ is given by $E=\mathcal{S}_u$ for some measurable function $u: \mathbb{R}^n \to \mathbb{R}$, then its nonlocal mean curvature can be written as
\begin{equation*}
H_s[E](x, u(x)) = 2\,\mathrm{p.v.} \int_{\mathbb{R}^n} G_s\left( \frac{u(x)-u(y)}{|x-y|} \right) \frac{\mathrm{d}y}{|x-y|^{n+s}} \eqqcolon \mathscr{H}_su(x)
\end{equation*}
at a point $x \in \mathbb{R}^n$, where
\begin{equation*}
G_s(t) = \int_0^t \frac{\mathrm{d}\tau}{(1+\tau^2)^{\frac{n+1+s}{2}}}.
\end{equation*}
See also Abatangelo--Valdinoci~\cite{AV14}, Barrios--Figalli--Valdinoci~\cite[Section~3]{BFV14}, and Bucur--Lombardini--Valdinoci~\cite[Appendix~B.1]{BLV19}. Thus, nonlocal minimal graphs satisfy the \emph{nonlocal minimal surface equation}
\begin{equation}\label{eq-NMSE}
\mathscr{H}_su(x) = 0.
\end{equation}
As described in \cite{CV13}, the minimal surface equation \eqref{eq-MSE} can be viewed as a local analogue of \eqref{eq-NMSE} in the limiting case $s=1$. For further details, we refer the reader to Caffarelli--Valdinoci~\cite{CV11} regarding the convergence of the $s$-perimeter to the classical perimeter as $s \nearrow 1$, and to Ambrosio--De Philippis--Martinazzi~\cite{ADPM11} for its Gamma-convergence.

The main result---the removable singularity theorem---extends to a broader class of equations, including those of prescribed nonlocal mean curvature. Specifically, we consider the equation
\begin{equation}\label{eq-main}
- \int_{\mathbb{R}^n} \mathscr{A}\left( x, y, \frac{u(x)-u(y)}{|x-y|} \right) \frac{\mathrm{d}y}{|x-y|^{n+s}} = \mathscr{B}(x, u),
\end{equation}
where $\mathscr{A}: \mathbb{R}^n \times \mathbb{R}^n \times \mathbb{R} \to \mathbb{R}$ and $\mathscr{B}: \mathbb{R}^n \times \mathbb{R} \to \mathbb{R}$ are Carath\'eodory functions satisfying $\mathscr{A}(x, y, 0)=0$ for a.e.\ $x, y \in \mathbb{R}^n$. We impose the following structural conditions on $\mathscr{A}$ and $\mathscr{B}$: We assume that $\mathscr{A}$ is monotone in the last variable and bounded, i.e.,
\begin{equation}\label{eq-A1}
\mathscr{A}(x, y, t_1) \leq \mathscr{A}(x, y, t_2) \quad\text{if }t_1<t_2, \quad \text{for a.e.\ $x, y \in \mathbb{R}^n$}
\end{equation}
and there exists a positive constant $\Lambda$ such that
\begin{equation}\label{eq-A2}
|\mathscr{A}(x, y, t)| \leq \Lambda \quad\text{for a.e.\ $x, y \in \mathbb{R}^n$ and for every $t \in \mathbb{R}$}.
\end{equation}
Since two conditions \eqref{eq-A1} and \eqref{eq-A2} do not exclude trivial function $\mathscr{A} \equiv 0$, we also assume that
\begin{equation}\label{eq-A3}
\inf_{x, y \in \mathbb{R}^n} \mathscr{A}(x, y, 1) > 0 \quad\text{and}\quad \sup_{x, y \in \mathbb{R}^n} \mathscr{A}(x, y, -1) < 0.
\end{equation}
Next, we assume that $\mathscr{B}$ satisfies
\begin{equation}\label{eq-B1}
\mathscr{B}(x, u(x)) \in L^1_{\mathrm{loc}}(G) \quad\text{whenever } u \in L^1_{\mathrm{loc}}(G)
\end{equation}
for any open set $G \subset \Omega$, and
\begin{equation}\label{eq-B2}
-(\mathrm{sign}\,z) \mathscr{B}(x, z) \leq f(x) \quad\text{for every } z \neq 0
\end{equation}
for some function $f \in L^{n/s}(\Omega)$. Throughout the paper, we assume \eqref{eq-A1}--\eqref{eq-B2} whenever $\mathscr{A}$ and $\mathscr{B}$ are involved.

\begin{example}\label{examples}
\begin{enumerate}[(i)]
\item
The equation of prescribed nonlocal mean curvature
\begin{equation*}
\mathscr{H}_su=h,
\end{equation*}
with $h \in L^{n/s}(\Omega)$, is a special case of \eqref{eq-main}. The constant nonlocal mean curvature equation corresponds to the case when $h$ is constant. Obviously, the nonlocal mean curvature equation \eqref{eq-NMSE} is covered.
\item
The condition \eqref{eq-B2} in the local case $s=1$ allows us to cover the equation of capillarity
\begin{equation*}
\mathrm{div}\left( \frac{Du}{\sqrt{1+|Du|^2}} \right) = \kappa u,
\end{equation*}
where $\kappa \in \mathbb{R}$ is a constant; see V\'azquez--V\'eron~\cite{VV81}. Similarly, the equation $\mathscr{H}_su=\kappa u$ is covered in the nonlocal framework.
\item
Cozzi--Lombardini~\cite{CL21} considered an even continuous function $g: \mathbb{R} \to (0,1]$ with $\lambda \coloneqq \int_0^\infty tg(t) \,\mathrm{d}t < \infty$ and its first antiderivative $G(t) = \int_0^t g(\tau)\,\mathrm{d}\tau$. Their motivation was to generalize the functions $(1+t^2)^{-(n+1+s)/s}$ and $G_s(t)$, as some of their results hold independently of the geometric structure. Note that $G$ satisfies \eqref{eq-A1}, \eqref{eq-A2}, and \eqref{eq-A3} with $\Lambda \coloneqq \int_0^\infty g(\tau)\,\mathrm{d}\tau \leq 1+\lambda<\infty$.
\end{enumerate}
\end{example}

There are several notions of solutions for \eqref{eq-NMSE} and \eqref{eq-main}. First, the classical (pointwise) solution $u$ of \eqref{eq-NMSE} requires certain regularity; for $\mathscr{H}_su(x)$ to be well-defined, the function $u$ has to be sufficiently regular near $x$ (e.g., $u \in C^{1, \alpha}$ for some $\alpha > s$ in a neighborhood of $x$). Alternatively, one may consider weak solutions of \eqref{eq-main}, defined as follows.

\begin{definition}\label{def-weak}
A measurable function $u: \mathbb{R}^n \to \mathbb{R}$ is a \emph{weak solution} of \eqref{eq-main} in $\Omega$ if
\begin{equation}\label{eq-main-weak}
\int_{\mathbb{R}^n} \int_{\mathbb{R}^n} \mathscr{A} \left( x, y, \frac{u(x)-u(y)}{|x-y|} \right) (\varphi(x)-\varphi(y)) \frac{\mathrm{d}y \,\mathrm{d}x}{|x-y|^{n+s}} + \int_{\Omega} \mathscr{B}(x, u)\varphi \,\mathrm{d}x = 0
\end{equation}
for any $\varphi \in C^\infty_c(\Omega)$.
\end{definition}

In sharp contrast to the weak solutions of the minimal surface equation \eqref{eq-MSE} in $\Omega$, which require to be in $W^{1, 1}_{\mathrm{loc}}(\Omega)$, Definition~\ref{def-weak} does not require any regularity on the weak solutions aside from measurability. However, it is still important to study weak solutions in the natural fractional Sobolev space $W^{s, 1}_{\mathrm{loc}}(\Omega)$ due to the following results: Cozzi--Lombardini~\cite[Theorem~1.10]{CL21} proved that $u \in W^{s, 1}_{\mathrm{loc}}(\Omega)$ is a weak solution of \eqref{eq-NMSE} in $\Omega$ if and only if $\partial\mathcal{S}_u$ is a nonlocal minimal graph in $\Omega \times \mathbb{R}$. Moreover, they also showed that this is equivalent to other solution concepts, such as pointwise solutions, viscosity solutions, and local minimizers of the corresponding nonlocal area functional. Therefore, it is meaningful to study the following geometric notion of solution, which actually corresponds to the nonlocal minimal graph in the special case $\mathscr{A}(x, y, t)=2G_s$ and $\mathscr{B}=0$.

\begin{definition}
A measurable function $u:\mathbb{R}^n \to \mathbb{R}$ is a \emph{solution} of \eqref{eq-main} in $\Omega$ if $u \in W^{s, 1}_{\mathrm{loc}}(\Omega)$ and $u$ is a weak solution of \eqref{eq-main} in $\Omega$.
\end{definition}

It is an interesting open question whether weak solutions are necessarily of class $W^{s, 1}_{\mathrm{loc}}$, even for \eqref{eq-NMSE}. We thus carefully distinguish between the notions of weak solution and solution throughout this work.

Our first main theorem is the removable singularity theorem for \eqref{eq-main}. Notably, unlike the nonlocal nonlinear $p$-Laplacian-type equations ($p >1$), the following theorem requires no assumptions on $u$ near $K$; see \cite[Theorem~1.1]{KL24} for comparison. For the definition of the $(s, 1)$-capacity used in the theorem, refer to Definition~\ref{def-cap}.

\begin{theorem}\label{thm-main}
Let $\Omega \subset \mathbb{R}^n$ be open and let $K \subset \Omega$ be a compact set of $(s, 1)$-capacity zero. If $u$ is a solution of \eqref{eq-main} in $\Omega \setminus K$, then $u$ has a representative that is a solution of \eqref{eq-main} in $\Omega$.
\end{theorem}

As mentioned above, in the special case where $\mathscr{A}(x, y, t)=2G_s$ and $\mathscr{B}=0$, any solution of \eqref{eq-NMSE} in $\Omega$ is actually in $C^\infty(\Omega)$ by the regularity theory for the nonlocal minimal graphs. However, to the best of the author's knowledge, a corresponding regularity theory for the more general equation \eqref{eq-main} has yet to be established. It is therefore presently unknown whether $u$ in Theorem~\ref{thm-main} has a continuous representative that is a solution of \eqref{eq-main} in $\Omega$.

A special case of Theorem~\ref{thm-main} for the equation \eqref{eq-NMSE} can be interpreted in terms of nonlocal minimal graphs as follows: \vspace{0.2cm}\\ 
\emph{Let $\Omega \subset \mathbb{R}^n$ be open and let $K \subset \Omega$ be a compact set of $(s, 1)$-capacity zero. Then, any nonlocal minimal graph in $\Omega \setminus K$ is indeed a nonlocal minimal graph in $\Omega$.} \vspace{0.2cm}

The proof of Theorem~\ref{thm-main} proceeds in three steps. In the first step, we show that the solution $u$ attains $W^{s, 1}_{\mathrm{loc}}(\Omega)$-regularity under an additional integrability assumption on $u$ in a neighborhood of $K$. To achieve this, we adapt the method in \cite{KL24}, where Kim--Lee extended Serrin's approach~\cite{Ser64} to fractional $p$-Laplacian-type equations with $1<p<\infty$. (Notice that the equation \eqref{eq-main} corresponds to the case $p=1$.) This method relies on Caccioppoli-type estimates involving double truncations of $u$ and the Moser iteration technique. Through this iterative process, the integrability of $u$ is improved to $L^\infty_{\mathrm{loc}}(\Omega)$, which in turn ensures $W^{s, 1}_{\mathrm{loc}}(\Omega)$-regularity.

The second step is to prove that the additional integrability condition on $u$ imposed in the first step is, in fact, satisfied by solutions. This requires a new idea, as the classical technique from Serrin~\cite{Ser65} is not directly applicable. In \cite{Ser65}, the double truncations were used again in a crucial way, with truncation levels greater than the maximum of $u$ on the boundary of $\Omega$. (Note that one can assume that $u$ is continuous on $\partial \Omega$ by considering a smaller open set.) This method does not work in our framework due to the nonlocality of the problem. The truncation levels need to exceed the supremum of $u$ on the complement of $\Omega$, but this supremum may not even be finite.

To address this issue, we employ a localization technique. This approach has been widely used for nonlocal equations such as fractional $p$-Laplacian-type equations and typically introduces the so-called nonlocal tail term, which is a weighted $L^{p-1}$-norm of $u$ outside $\Omega$. Since the equation \eqref{eq-main} corresponds to the case $p=1$, it involves a form of `weighted $L^0$-norm' of $u$, which is constant. This enables us to adapt Serrin's proof~\cite{Ser65} using double truncations.

In the final step, we use a standard approximation argument to show that $u$ has a representative that is a solution of \eqref{eq-main} in $\Omega$.

We conclude the introduction with a few remarks. Another possible direction for generalizing the classical result of Bers~\cite{Ber51} involves noncompact set $K$. Specifically, for an open set $\Omega \subset \mathbb{R}^n$ and a relatively closed set $E \subset \Omega$ of $(n-1)$-dimensional Hausdorff measure zero, any continuous weak solution of \eqref{eq-MSE} in $\Omega \setminus E$ extends to a continuous weak solution of the same equation \eqref{eq-MSE} in $\Omega$. This result was first proved by Nitsche~\cite{Nit65} for $n=2$, and the general case in higher dimensions was established by Miranda~\cite{Mir77} and Simon~\cite{Sim77}. See also Lau~\cite{Lau88}.

It remains an intriguing open question whether Theorem~\ref{thm-main} still holds when $K$ approaches the boundary of $\Omega$. This problem appears to be challenging and seems to require new ideas. While we do not address it in this paper, we do provide some insight for weak solutions in Theorem~\ref{thm-main-weak}.

The author would like to thank the anonymous referees for their careful reading of the manuscript and for providing helpful comments.

This article is organized as follows. After presenting some preliminary material in Section~\ref{sec-preliminaries}, we provide Caccioppoli-type estimates involving the double truncations in Section~\ref{sec-Caccio}. The main theorem, Theorem~\ref{thm-main}, is proved in Section~\ref{sec-RST}. Finally, the last section, Section~\ref{sec-weak}, is devoted to the removable singularity theorem for weak solutions.

\section{preliminaries}\label{sec-preliminaries}

In this section, we study basic properties of solutions of \eqref{eq-NMSE} and \eqref{eq-main}, review the fractional Sobolev inequality, and collect several facts about the $(s, 1)$-capacity.

We begin with the following proposition, which shows that the equations \eqref{eq-NMSE} and \eqref{eq-main} allow a larger class of test functions than $C^\infty_c(\Omega)$.

\begin{proposition}\label{prop-test}
Let $u: \mathbb{R}^n \to \mathbb{R}$ be a measurable function such that $u \in W^{s, 1}_{\mathrm{loc}}(\Omega)$. Then, $u$ is a solution of \eqref{eq-main} in $\Omega$ if and only if \eqref{eq-main-weak} holds for all $\varphi \in W^{s, 1}_{\mathrm{loc}}(\Omega) \cap L^\infty_{\mathrm{loc}}(\Omega)$ with $\mathrm{supp}\,\varphi \Subset \Omega$.
\end{proposition}

\begin{proof}
One implication is trivial. Suppose that $\varphi \in W^{s, 1}_{\mathrm{loc}}(\Omega) \cap L^\infty_{\mathrm{loc}}(\Omega)$ and that $\mathrm{supp}\,\varphi \Subset \Omega$. Let $\Omega_1$ and $\Omega_2$ be open sets such that $\mathrm{supp}\,\varphi \Subset \Omega_2 \Subset \Omega_1 \Subset \Omega$. By mollification, there are functions $\varphi_j \in C^\infty_c(\Omega_2)$ such that $\varphi_j \to \varphi$ in $W^{s, 1}(\Omega_1)$ and in $L^\infty(\Omega_1)$ as $j \to \infty$. Then, using \eqref{eq-A2}, we obtain that
\begin{align*}
&\left| \int_{\mathbb{R}^n} \int_{\mathbb{R}^n} \mathscr{A} \left( x, y, \frac{u(x)-u(y)}{|x-y|} \right) ((\varphi_j-\varphi)(x)-(\varphi_j-\varphi)(y)) \frac{\mathrm{d}y \,\mathrm{d}x}{|x-y|^{n+s}} \right| \\
&\leq \Lambda [\varphi_j-\varphi]_{W^{s, 1}(\Omega_1)} + 2\Lambda \int_{\Omega_1} \int_{\mathbb{R}^n \setminus \Omega_1} \frac{|\varphi_j(x)-\varphi(x)|}{|x-y|^{n+s}} \,\mathrm{d}y \,\mathrm{d}x \\
&\leq \Lambda [\varphi_j-\varphi]_{W^{s, 1}(\Omega_1)} + 2\Lambda \int_{\Omega_2} |\varphi_j(x)-\varphi(x)| \int_{\mathbb{R}^n \setminus B_d(x)} \frac{\mathrm{d}y}{|x-y|^{n+s}} \,\mathrm{d}x \\
&\leq \Lambda [\varphi_j-\varphi]_{W^{s, 1}(\Omega_1)} + 2\Lambda \frac{|\mathbb{S}^{n-1}|}{sd^s} \|\varphi_j-\varphi\|_{L^1(\Omega_2)},
\end{align*}
where $d=\mathrm{dist}(\Omega_2, \Omega_1^c)>0$, and that
\begin{equation*}
\left| \int_{\Omega} \mathscr{B}(x, u)(\varphi_j-\varphi) \,\mathrm{d}x \right| \leq \|\varphi_j-\varphi\|_{L^\infty(\Omega_2)} \int_{\Omega_2} |\mathscr{B}(x, u)| \,\mathrm{d}x.
\end{equation*}
Notice that $\mathscr{B}(x, u) \in L^1(\Omega_2)$ by the assumption \eqref{eq-B1}. Therefore, the weak formulation \eqref{eq-main-weak} follows from \eqref{eq-main-weak} with $\varphi_j$ in place of $\varphi$, by taking the limit as $j \to \infty$.
\end{proof}

One of the key tools in our analysis is the fractional Sobolev inequality. Numerous variations of fractional Sobolev inequalities can be found in the literature (see, for instance, Maz'ya--Shaposhnikova~\cite[Theorem~1]{MS02} and Di Nezza--Palatucci--Valdinoci~\cite[Theorem~6.7]{DNPV12}), but the following form will be particularly useful in this work.

\begin{theorem}\label{thm-FSI}
Let $G \subset \mathbb{R}^n$ be a bounded open set with a Lipschitz boundary. Then, there exists a constant $C=C(n, s, G)>0$ such that
\begin{equation*}
\|u\|_{L^{1^{\ast}_s}(G)} \leq C \|u\|_{W^{s, 1}(G)}
\end{equation*}
for any $u \in W^{s, 1}(G)$, where
\begin{equation}\label{eq-1-ast}
1^{\ast}_s=\frac{n}{n-s}.
\end{equation}
In particular, there exists a constant $C=C(n, s)>0$ such that
\begin{equation*}
\|u\|_{L^{1^{\ast}_s}(B_r(x_0))} \leq C \left( [u]_{W^{s, 1}(B_r(x_0))} + r^{-s} \|u\|_{L^1(B_r(x_0))} \right)
\end{equation*}
for any $u \in W^{s, 1}(B_r(x_0))$.
\end{theorem}

The remainder of this section is devoted to the definition and some properties of the $(s, 1)$-capacity. While the main theorems in the introduction require the $(s, 1)$-capacity only for compacts sets, we will need it for general sets in the sequel. Thus, we define it for general sets here.

\begin{definition}\label{def-cap}
Let $\Omega \subset \mathbb{R}^n$ be open. The \emph{$(s, 1)$-capacity} of a compact set $K \subset \Omega$ is defined by
\begin{equation*}
\mathrm{cap}_{s, 1}(K, \Omega) \coloneqq \inf_{u} {[u]_{W^{s, 1}(\mathbb{R}^n)}},
\end{equation*}
where the infimum is taken over all $u \in C^\infty_c(\Omega)$ such that $u \geq 1$ on $K$. For open sets $G \subset \Omega$,
\begin{equation*}
\mathrm{cap}_{s, 1}(G,\Omega) \coloneqq \sup_{\substack{K \text{ compact}\\ K \subset G}} \mathrm{cap}_{s, 1}(K,\Omega),
\end{equation*}
and for arbitrary sets $E \subset \Omega$,
\begin{equation*}
\mathrm{cap}_{s, 1}(E,\Omega) \coloneqq \inf_{\substack{G \text{ open}\\ E \subset G \subset \Omega}} \mathrm{cap}_{s, 1}(G,\Omega).
\end{equation*}
\end{definition}

The $(s, 1)$-capacity is simply the $(s, p)$-capacity with $p=1$. However, the study of $(s, p)$-capacity in the literature has primarily been restricted to the case $p>1$; see, for instance, Bj\"orn--Bj\"orn--Kim~\cite[Section~5]{BBK24} and Kim--Lee~\cite[Section~2.3]{KL24}. Nevertheless, most of the their results remain valid even for $p=1$ without modification of the proofs. We summarize these results below.

\begin{lemma}\label{lem-cap}\cite[Lemma~2.17]{KL24}
Let $E \subset \mathbb{R}^n$. Then,
\begin{equation*}
\mathrm{cap}_{s, 1}(E, \mathbb{R}^n) = \inf_{u} {[u]_{W^{s, 1}(\mathbb{R}^n)}},
\end{equation*}
where the infimum is taken over all $u \in W^{s, 1}(\mathbb{R}^n)$ such that $u=1$ in a neighborhood of $E$ and $0 \leq u \leq 1$ everywhere.
\end{lemma}

Before stating the next results, let us define the set of $(s, 1)$-capacity zero, which plays a fundamental role in the main theorems.

\begin{definition}
A set $E \subset \mathbb{R}^n$ is of \emph{$(s, 1)$-capacity zero} if $\mathrm{cap}_{s, 1}(E \cap \Omega, \Omega)=0$ for all open sets $\Omega \subset \mathbb{R}^n$.
\end{definition}

The next result will be used frequently throughout the paper.

\begin{lemma}\label{lem-cap-zero}
Let $E \subset \mathbb{R}^n$ be of $(s, 1)$-capacity zero. Then, there exists a sequence of functions $\bar{\eta}_j \in W^{s, 1}(\mathbb{R}^n)$ such that $\bar{\eta}_j=0$ in a neighborhood of $E$, $0 \leq \bar{\eta}_j \leq 1$ in $\mathbb{R}^n$, $\lim_{j\to \infty} [\bar{\eta}_j]_{W^{s, 1}(\mathbb{R}^n)} = 0$, and $\bar{\eta}_j \to 1$ a.e.\ in $\mathbb{R}^n$ as $j \to \infty$.
\end{lemma}

\begin{proof}
Lemma~\ref{lem-cap} shows that there exists a sequence of functions $\eta_j \in W^{s, 1}(\mathbb{R}^n)$ such that $\eta_j=1$ in a neighborhood of $E$, $0 \leq \eta_j \leq 1$ in $\mathbb{R}^n$, and $[\eta_j]_{W^{s, 1}(\mathbb{R}^n)} \to 0$ as $j \to \infty$. Note that we may assume that $\eta_j \to 0$ a.e.\ in $\mathbb{R}^n$ as $j \to \infty$ by taking a subsequence if necessary. Then, the functions $\bar{\eta}_j=1-\eta_j$ have the desired properties.
\end{proof}

We close this section with the following result, which shows that sets of $(s, 1)$-capacity zero have measure zero.

\begin{lemma}\label{lem-measure}\cite[Lemma~2.15]{KL24}
If $E \subset \mathbb{R}^n$ is of $(s, 1)$-capacity zero, then $|E|=0$.
\end{lemma}

\section{Caccioppoli-type estimates}\label{sec-Caccio}

As explained in the introduction, the first and second steps of the proof of the removable singularity theorem rely on a special type of Caccioppoli estimates involving the double truncations $\bar{u}$ and $\underbar{$u$}$ of $u$. For $l>k\geq0$, we define
\begin{equation}\label{eq-truncation}
\bar{u}=
\begin{cases}
k &\text{if } u \leq k, \\
u &\text{if } k < u < l, \\
l &\text{if } u \geq l,
\end{cases}
\quad\text{and}\quad
\underbar{$u$}=
\begin{cases}
l &\text{if } u \leq -l, \\
-u &\text{if } -l < u < -k, \\
k &\text{if } u \geq -k.
\end{cases}
\end{equation}
Note that if $k=0$, then $\bar{u}=\min\{u_+, l\}$ and $\underbar{$u$}=\min\{u_-, l\}$, where $u_+\coloneqq \max\{u, 0\}$ and $u_-\coloneqq -\min\{u, 0\}$ denote the positive and negative parts of $u$, respectively.

Before we state the Caccioppoli estimates, we observe that the conditions \eqref{eq-A1} and \eqref{eq-A3} imply the existence of a constant $\lambda > 0$ such that
\begin{equation}\label{eq-A-lower}
p\mathscr{A}(x, y, p) \geq \lambda (|p| - 1).
\end{equation}
Indeed, one can take $\lambda = \min\{\inf_{x, y \in \mathbb{R}^n} \mathscr{A}(x, y, 1), -\sup_{x, y \in \mathbb{R}^n} \mathscr{A}(x, y, -1)\}$ and prove \eqref{eq-A-lower} by considering the cases $|p| \leq 1$ and $|p|>1$ separately.

We now present the Caccioppoli-type estimates. It is worth noting that in most of the subsequent results, the set $E$ is not necessarily a compact subset of $\Omega$.

\begin{lemma}\label{lem-Caccio1}
Let $\Omega \subset \mathbb{R}^n$ be open and let $E \subset \Omega$ be a relatively closed set of $(s, 1)$-capacity zero. Assume that $u$ is a solution of \eqref{eq-main} in $\Omega \setminus E$. Let $G$ be a nonempty bounded open subset of $\Omega$ and let $\eta \in C^\infty_c(G)$ be such that $0 \leq \eta \leq 1$ in $G$. Let $\bar{\eta} \in W^{s, 1}(\mathbb{R}^n)$ be a function such that $\bar{\eta}$ vanishes in a neighborhood of $E$ and $0 \leq \bar{\eta} \leq 1$ in $\mathbb{R}^n$. Then, for any $l>k\geq 0$ and $\beta>0$,
\begin{equation}\label{eq-Caccio1-sub}
\begin{split}
&\int_{G} \int_{G} \frac{|\bar{v}^\beta(x) - \bar{v}^\beta(y)|}{|x-y|^{n+s}} (\eta\bar{\eta})(y) \,\mathrm{d}y \,\mathrm{d}x \\
&\leq C \int_{G} \int_{G} \bar{v}^\beta(x) \frac{|(\eta\bar{\eta})(x)-(\eta\bar{\eta})(y)|}{|x-y|^{n+s}} \,\mathrm{d}y \,\mathrm{d}x \\
&\quad + C \left( \beta (\mathrm{diam}\, G)^{1-s} + \sup_{x \in \mathrm{supp}\,\eta} \int_{\mathbb{R}^n \setminus G} \frac{\mathrm{d}y}{|x-y|^{n+s}} \right) \int_G \bar{v}^\beta \,\mathrm{d}x \\
&\quad + C \|f\|_{L^{n/s}(G \cap \{u>k\})} \|\bar{v}^\beta\eta\bar{\eta}\|_{L^{1^{\ast}_s}(G)}
\end{split}
\end{equation}
and
\begin{equation}\label{eq-Caccio1-super}
\begin{split}
&\int_{G} \int_{G} \frac{|\underbar{$v$}^\beta(x) - \underbar{$v$}^\beta(y)|}{|x-y|^{n+s}} (\eta\bar{\eta})(y) \,\mathrm{d}y \,\mathrm{d}x \\
&\leq C \int_{G} \int_{G} \underbar{$v$}^\beta(x) \frac{|(\eta\bar{\eta})(x)-(\eta\bar{\eta})(y)|}{|x-y|^{n+s}} \,\mathrm{d}y \,\mathrm{d}x \\
&\quad + C \left( \beta (\mathrm{diam}\, G)^{1-s} + \sup_{x \in \mathrm{supp}\,\eta} \int_{\mathbb{R}^n \setminus G} \frac{\mathrm{d}y}{|x-y|^{n+s}} \right) \int_G \underbar{$v$}^\beta \,\mathrm{d}x \\
&\quad + C \|f\|_{L^{n/s}(G \cap \{u<-k\})} \|\underbar{$v$}^\beta\eta\bar{\eta}\|_{L^{1^{\ast}_s}(G)},
\end{split}
\end{equation}
where $\bar{v}=\bar{u}-k+\lambda$, $\underbar{$v$}=\underbar{$u$}-k+\lambda$, $C=C(n, s, \Lambda, \lambda)>0$, and $1^{\ast}_s$ is given by \eqref{eq-1-ast}.
\end{lemma}

\begin{proof}
We define $H: [k, \infty) \to [0, \infty)$ by
\begin{equation*}
H(t) = (t-k+\lambda)^\beta-\lambda^\beta
\end{equation*}
and set $\varphi=H(\bar{u}) \eta\bar{\eta}$. Then, $\varphi \in W^{s, 1}_{\mathrm{loc}}(\Omega \setminus E) \cap L^\infty(\Omega \setminus E)$ and $\mathrm{supp}\,\varphi \Subset \Omega \setminus E$. It thus follows from Proposition~\ref{prop-test} that
\begin{align*}
0
&= \int_{G} \int_{G} \mathscr{A} \left( x, y, \frac{u(x)-u(y)}{|x-y|} \right) (H(\bar{u}(x)) - H(\bar{u}(y))) (\eta\bar{\eta})(y) \, \frac{\mathrm{d}y \,\mathrm{d}x}{|x-y|^{n+s}} \\
&\quad + \int_{G} \int_{G} \mathscr{A} \left( x, y, \frac{u(x)-u(y)}{|x-y|} \right) H(\bar{u}(x)) ((\eta\bar{\eta})(x) - (\eta\bar{\eta})(y)) \, \frac{\mathrm{d}y \,\mathrm{d}x}{|x-y|^{n+s}} \\
&\quad + \iint_{(G \times G^c) \cup (G^c \times G)} \mathscr{A} \left( x, y, \frac{u(x)-u(y)}{|x-y|} \right) (\varphi(x)-\varphi(y)) \, \frac{\mathrm{d}y \,\mathrm{d}x}{|x-y|^{n+s}} \\
&\quad + \int_{\Omega \setminus E} \mathscr{B}(x, u)\varphi \,\mathrm{d}x \\
&\eqqcolon I_1 + I_2 + I_3 + I_4.
\end{align*}

In order to estimate $I_1$, we first observe that
\begin{equation}\label{eq-claim-ubar}
\begin{split}
&\mathscr{A} \left( x, y, \frac{u(x)-u(y)}{|x-y|} \right) (H(\bar{u}(x)) - H(\bar{u}(y))) \\
&\geq \mathscr{A} \left( x, y, \frac{\bar{u}(x)-\bar{u}(y)}{|x-y|} \right) (H(\bar{u}(x)) - H(\bar{u}(y))).
\end{split}
\end{equation}
Indeed, if $u(x) \geq u(y)$, then $u(x)-u(y) \geq \bar{u}(x)-\bar{u}(y)$, and hence \eqref{eq-claim-ubar} follows from the assumption \eqref{eq-A1} and the monotonicity of $H$. The remaining case $u(x) < u(y)$ can be treated in the same way.

We next obtain by using \eqref{eq-A-lower} and $\lambda H'(t) \leq \beta (t-k+\lambda)^\beta$ that
\begin{equation}\label{eq-Caccio-I1}
\begin{split}
&\mathscr{A} \left( x, y, \frac{\bar{u}(x)-\bar{u}(y)}{|x-y|} \right) (H(\bar{u}(x)) - H(\bar{u}(y))) \\
&= \mathscr{A} \left( x, y, \frac{\bar{u}(x)-\bar{u}(y)}{|x-y|} \right) \frac{\bar{u}(x)-\bar{u}(y)}{|x-y|} |x-y| \fint_{\bar{u}(y)}^{\bar{u}(x)} H'(t) \,\mathrm{d}t \\
&\geq \lambda \left( \frac{|\bar{u}(x)-\bar{u}(y)|}{|x-y|} - 1 \right) |x-y| \fint_{\bar{u}(y)}^{\bar{u}(x)} H'(t) \,\mathrm{d}t \\
&\geq \lambda |H(\bar{u}(x)) - H(\bar{u}(y))| - \beta \max\{ \bar{v}^\beta(x), \bar{v}^\beta(y) \} |x-y|.
\end{split}
\end{equation}
Combining \eqref{eq-claim-ubar} and \eqref{eq-Caccio-I1} yields that
\begin{align*}
I_1
&\geq \lambda \int_{G} \int_{G} \frac{|H(\bar{u}(x))-H(\bar{u}(y))|}{|x-y|^{n+s}} (\eta\bar{\eta})(y) \,\mathrm{d}y \,\mathrm{d}x \\
&\quad - \beta \int_{G} \int_{G} \frac{\max\{ \bar{v}^\beta(x), \bar{v}^\beta(y) \}}{|x-y|^{n-(1-s)}} \,\mathrm{d}y \,\mathrm{d}x \\
&\geq \lambda \int_{G} \int_{G} \frac{|\bar{v}^\beta(x)-\bar{v}^\beta(y)|}{|x-y|^{n+s}} (\eta\bar{\eta})(y) \,\mathrm{d}y \,\mathrm{d}x - \beta \frac{|\mathbb{S}^{n-1}|(\mathrm{diam}\, G)^{1-s}}{1-s} \int_{G} \bar{v}^\beta \,\mathrm{d}x.
\end{align*}

For $I_2$ and $I_3$, we use the condition \eqref{eq-A2} to obtain that
\begin{equation*}
I_2
\geq - \Lambda \int_{G} \int_{G} \bar{v}^\beta(x) \frac{|(\eta\bar{\eta})(x)-(\eta\bar{\eta})(y)|}{|x-y|^{n+s}} \,\mathrm{d}y \,\mathrm{d}x
\end{equation*}
and that
\begin{align*}
I_3
&\geq -\Lambda \int_{G} \int_{G^c} \frac{\bar{v}^\beta(x) (\eta\bar{\eta})(x)}{|x-y|^{n+s}} \,\mathrm{d}y\,\mathrm{d}x - \Lambda \int_{G^c} \int_{G} \frac{\bar{v}^\beta(y) (\eta\bar{\eta})(y)}{|x-y|^{n+s}} \,\mathrm{d}y \,\mathrm{d}x \\
&\geq -2\Lambda \left( \sup_{x \in \mathrm{supp}\,\eta} \int_{\mathbb{R}^n \setminus G} \frac{\mathrm{d}y}{|x-y|^{n+s}} \right) \int_{G} \bar{v}^\beta \,\mathrm{d}x.
\end{align*}
Finally, we observe that the set $E$ has measure zero by Lemma~\ref{lem-measure} and that $H(\bar{u}(x))=0$ if $u(x) \leq k$. Thus, $I_4$ can be estimated as
\begin{equation*}
I_4 \geq - \int_{G \cap \{u>k\}} fH(\bar{u}) \eta \bar{\eta} \,\mathrm{d}x \geq - \|f\|_{L^{n/s}(G \cap \{u>k\})} \|\bar{v}^\beta \eta\bar{\eta}\|_{L^{1^{\ast}_s}(G)}
\end{equation*}
by using the assumption \eqref{eq-B2}. Combining all the estimates above conclude the desired estimate \eqref{eq-Caccio1-sub}.

The other estimate \eqref{eq-Caccio1-super} can be obtained by using $\varphi=H(\underbar{$u$})\eta \bar{\eta}$ as a test function and following the proof above.
\end{proof}

We simplify Lemma~\ref{lem-Caccio1} by taking a particular $\eta$ and using the fractional Sobolev inequality in Theorem~\ref{thm-FSI}.

\begin{lemma}\label{lem-Caccio2}
Let $\Omega \subset \mathbb{R}^n$ be open and let $E \subset \Omega$ be a relatively closed set of $(s, 1)$-capacity zero. Suppose that $u$ is a solution of \eqref{eq-main} in $\Omega \setminus E$. Let $\bar{\eta} \in W^{s, 1}(\mathbb{R}^n)$ be a function such that $\bar{\eta}$ vanishes in a neighborhood of $E$ and $0 \leq \bar{\eta} \leq 1$ in $\mathbb{R}^n$. Then, for each $x_0 \in \Omega$, there exists $R \in (0,1)$ such that $B_R=B_R(x_0) \subset \Omega$ and that for any $l, \beta>0$ and $0<\rho<r\leq R$,
\begin{align*}
&\int_{B_\rho} \int_{B_\rho} \frac{|w_\pm^\beta(x) - w_\pm^\beta(y)|}{|x-y|^{n+s}} \bar{\eta}(y) \,\mathrm{d}y \,\mathrm{d}x \\
&\leq C \int_{B_r} \int_{B_r} w_\pm^\beta(x) \frac{|\bar{\eta}(x)-\bar{\eta}(y)|}{|x-y|^{n+s}} \,\mathrm{d}y \,\mathrm{d}x + C(1+\beta) \frac{r}{r-\rho} r^{-s} \int_{B_r} w_\pm^\beta \,\mathrm{d}x,
\end{align*}
where $w_\pm\coloneqq \min\{u_\pm, l\}+\lambda$ and $C=C(n, s, \Lambda, \lambda)>0$.
\end{lemma}

\begin{proof}
Let $\eta \in C^\infty_c(B_{(r+\rho)/2})$ be such that $\eta = 1$ on $B_\rho$, $0 \leq \eta \leq 1$, and $|\nabla \eta| \leq 4/(r-\rho)$. Then, Lemma~\ref{lem-Caccio1} with $G=B_r$ and $k=0$ shows that
\begin{equation}\label{eq-Caccio2-1}
\begin{split}
I
&\coloneqq \int_{B_r} \int_{B_r} \frac{|w_\pm^\beta(x) - w_\pm^\beta(y)|}{|x-y|^{n+s}} (\eta\bar{\eta})(y) \,\mathrm{d}y \,\mathrm{d}x \\
&\leq C_1 \int_{B_r} \int_{B_r} w_\pm^\beta(x) \frac{|(\eta\bar{\eta})(x)-(\eta\bar{\eta})(y)|}{|x-y|^{n+s}} \,\mathrm{d}y \,\mathrm{d}x \\
&\quad + C_1 \left( \beta r^{1-s} + \int_{\mathbb{R}^n \setminus B_{(r-\rho)/2}(x)} \frac{\mathrm{d}y}{|x-y|^{n+s}} \right) \int_{B_r} w_\pm^\beta \,\mathrm{d}x \\
&\quad + C_1 \|f\|_{L^{n/s}(B_R)} \|w_\pm^\beta \eta\bar{\eta}\|_{L^{1^{\ast}_s}(B_r)}
\end{split}
\end{equation}
for some $C_1=C_1(n, s, \Lambda, \lambda)>0$. By the fractional Sobolev inequality in Theorem~\ref{thm-FSI}, we have that
\begin{equation}\label{eq-Caccio2-2}
\begin{split}
\|w_\pm^\beta\eta\bar{\eta}\|_{L^{1^{\ast}_s}(B_r)}
&\leq C_2 \left( [w_\pm^\beta\eta\bar{\eta}]_{W^{s, 1}(B_r)} + r^{-s} \|w_\pm^\beta\eta\bar{\eta}\|_{L^1(B_r)} \right) \\
&\leq C_2 I + C_2 \int_{B_r} \int_{B_r} w_\pm^\beta(x) \frac{|(\eta\bar{\eta})(x)-(\eta\bar{\eta})(y)|}{|x-y|^{n+s}} \,\mathrm{d}y \,\mathrm{d}x \\
&\quad + C_2 r^{-s} \int_{B_r} w_\pm^\beta \eta\bar{\eta} \,\mathrm{d}x
\end{split}
\end{equation}
for some $C_2=C_2(n, s)>0$. We now take $R \in (0, 1)$ sufficiently small so that
\begin{equation}\label{eq-absorption}
C_1 C_2 \|f\|_{L^{n/s}(B_R)} \leq 1/2.
\end{equation}
We combine \eqref{eq-Caccio2-1}--\eqref{eq-Caccio2-2} and then use \eqref{eq-absorption} to obtain that
\begin{align*}
\frac{1}{2}I
&\leq C \int_{B_r} \int_{B_r} w_\pm^\beta(x) \frac{|(\eta\bar{\eta})(x)-(\eta\bar{\eta})(y)|}{|x-y|^{n+s}} \,\mathrm{d}y \,\mathrm{d}x \\
&\quad + C \left( \beta r^{1-s} + (r-\rho)^{-s} + r^{-s} \right) \int_{B_r} w_\pm^\beta \,\mathrm{d}x,
\end{align*}
where $C=C(n, s, \Lambda, \lambda)>0$. Since
\begin{align*}
\int_{B_r} \frac{|(\eta\bar{\eta})(x)-(\eta\bar{\eta})(y)|}{|x-y|^{n+s}} \,\mathrm{d}y
&\leq \int_{B_r} \frac{|\bar{\eta}(x)-\bar{\eta}(y)|}{|x-y|^{n+s}} \,\mathrm{d}y + \int_{B_r} \frac{|\eta(x)-\eta(y)|}{|x-y|^{n+s}} \,\mathrm{d}y \\
&\leq \int_{B_r} \frac{|\bar{\eta}(x)-\bar{\eta}(y)|}{|x-y|^{n+s}} \,\mathrm{d}y + C\frac{r^{1-s}}{r-\rho}
\end{align*}
and
\begin{equation*}
\beta r^{1-s} + (r-\rho)^{-s} + r^{-s} + \frac{r^{1-s}}{r-\rho} \leq C(1+\beta) \frac{r}{r-\rho}r^{-s},
\end{equation*}
the desired estimate follows.
\end{proof}

\section{Removable singularity theorem}\label{sec-RST}

In this section, we provide the proof of the main theorem, Theorem~\ref{thm-main}. As detailed in the introduction, the proof involves three steps. The first step establishes that the solution, initially assumed to belong to $W^{s, 1}_{\mathrm{loc}}(\Omega \setminus K)$, attains $W^{s, 1}_{\mathrm{loc}}(\Omega)$-regularity under the additional integrability condition $u \in L^\varepsilon_{\mathrm{loc}}(\Omega)$ for some $\varepsilon>0$. Since the compactness of $K$ is not required in this step, the following lemma is stated for noncompact removable sets.

\begin{lemma}\label{lem-step1}
Let $\Omega \subset \mathbb{R}^n$ be open and let $E \subset \Omega$ be a relatively closed set of $(s, 1)$-capacity zero. If $u$ is a solution of \eqref{eq-main} in $\Omega \setminus E$ and $u \in L^{\varepsilon}_{\mathrm{loc}}(\Omega)$ for some $\varepsilon > 0$, then $u \in W^{s, 1}_{\mathrm{loc}}(\Omega)$.
\end{lemma}

\begin{proof}
We fix $x_0 \in \Omega$ and let $R \in (0,1)$ be the radius given in Lemma~\ref{lem-Caccio2}. It is enough to show that $u \in W^{s, 1}(B_{R/4})$. Since $E$ has $(s, 1)$-capacity zero, Lemma~\ref{lem-cap-zero} shows that there exists a sequence of functions $\bar{\eta}_j \in W^{s, 1}(\mathbb{R}^n)$ such that $\bar{\eta}_j=0$ in a neighborhood of $E$, $0 \leq \bar{\eta}_j \leq 1$ in $\mathbb{R}^n$, $\lim_{j\to \infty} [\bar{\eta}_j]_{W^{s,1}(\mathbb{R}^n)} \to 0$, and $\bar{\eta}_j \to 1$ a.e.\ in $\mathbb{R}^n$ as $j \to \infty$. Suppose $R/4\leq \rho<r\leq R$ and let $l, \beta>0$. Then, Lemma~\ref{lem-Caccio2} with $\bar{\eta}_j$ in place of $\bar{\eta}$ shows that
\begin{equation*}
\begin{split}
&\int_{B_\rho} \int_{B_\rho} \frac{|w_\pm^\beta(x)-w_\pm^\beta(y)|}{|x-y|^{n+s}} \bar{\eta}_j(y) \,\mathrm{d}y \,\mathrm{d}x \\
&\leq C \int_{B_r} \int_{B_r} w_\pm^\beta(x) \frac{|\bar{\eta}_j(x)-\bar{\eta}_j(y)|}{|x-y|^{n+s}} \,\mathrm{d}y\,\mathrm{d}x + C\frac{1+\beta}{r-\rho} \int_{B_r} w_\pm^\beta \,\mathrm{d}x \\
&\leq C (l+\lambda)^\beta [\bar{\eta}_j]_{W^{s, 1}(B_r)} + C \frac{1+\beta}{r-\rho} \int_{B_r} w_\pm^\beta \,\mathrm{d}x,
\end{split}
\end{equation*}
where $w_\pm=\min\{u_\pm, l\} + \lambda$ and $C=C(n, s, \Lambda, \lambda)>0$.

Since $\lim_{j \to \infty} [\bar{\eta}_j]_{W^{s, 1}(B_r)} = 0$, taking the limit as $j \to \infty$, we obtain that
\begin{equation}\label{eq-I2}
\begin{split}
[w_\pm^\beta]_{W^{s, 1}(B_\rho)}
&\leq \liminf_{j \to \infty} \int_{B_\rho} \int_{B_\rho} \frac{|w_\pm^\beta(x)-w_\pm^\beta(y)|}{|x-y|^{n+s}} \bar{\eta}_j(y) \,\mathrm{d}y \,\mathrm{d}x \\
&\leq C \frac{1+\beta}{r-\rho} \int_{B_r} w_\pm^\beta \,\mathrm{d}x.
\end{split}
\end{equation}
The estimate \eqref{eq-I2}, combined with the fractional Sobolev inequality in Theorem~\ref{thm-FSI}, yields that
\begin{equation*}
\left( \fint_{B_\rho} w_\pm^{\beta 1^{\ast}_s} \,\mathrm{d}x \right)^{1/1^{\ast}_s} \leq C \rho^{s-n} [w_\pm^\beta]_{W^{s, 1}(B_\rho)} + C \fint_{B_\rho} w_\pm^\beta \,\mathrm{d}x \leq C \frac{1+\beta}{r-\rho} \fint_{B_r} w_\pm^\beta \,\mathrm{d}x.
\end{equation*}
Since $w_\pm \to u_\pm+\lambda$ as $l \to \infty$, taking the limit as $l \to \infty$ shows that
\begin{equation}\label{eq-Phi-gamma}
\Phi_\pm(\beta 1^{\ast}_s, \rho) \leq \left( C \frac{1+\beta}{r-\rho} \right)^{1/\beta} \Phi_\pm(\beta, r),
\end{equation}
where
\begin{equation*}
\Phi_\pm(\beta, r) = \left( \fint_{B_r} (u_\pm+\lambda)^{\beta} \,\mathrm{d}x \right)^{1/\beta}.
\end{equation*}
A standard iteration of \eqref{eq-Phi-gamma} beginning with $\beta=\varepsilon$ establishes that
\begin{equation}\label{eq-bdd}
\|u_\pm+\lambda\|_{L^\infty(B_{R/2})} \leq C \Phi(\varepsilon, R) = C R^{-n/\varepsilon} \|u_\pm+\lambda\|_{L^{\varepsilon}(B_{R})} < \infty.
\end{equation}
Moreover, applying the estimate \eqref{eq-I2} with $\beta=1$, $\rho=R/4$, and $r=R/2$ as $l \to \infty$, along with the estimate \eqref{eq-bdd}, shows that
\begin{equation*}
[u_\pm]_{W^{s, 1}(B_{R/4})} \leq \frac{C}{R} \int_{B_{R/2}} (u_\pm+\lambda) \,\mathrm{d}x \leq CR^{n-1} \|u_\pm+\lambda\|_{L^\infty(B_{R/2})} < \infty.
\end{equation*}
This concludes that $u \in W^{s, 1}(B_{R/4})$.
\end{proof}

The second step in the proof of Theorem~\ref{thm-main} is to show that the solution attains a certain level of integrability, which allows the additional assumption on $u$ imposed in Lemma~\ref{lem-step1} to be removed. We remark that this is the only part of the proof where the compactness of $K \subset \Omega$ is used.

\begin{lemma}\label{lem-step2}
Let $\Omega \subset \mathbb{R}^n$ be open and let $K \subset \Omega$ be a compact set of $(s, 1)$-capacity zero. If $u$ is a solution of \eqref{eq-main} in $\Omega \setminus K$, then $u \in L^{1^{\ast}_s}_{\mathrm{loc}}(\Omega)$.
\end{lemma}

\begin{proof}
Let $\Omega' \Subset \Omega$ be an open set such that $K \subset \Omega'$. It is enough to show that $u \in L^{1^{\ast}_s}(\Omega')$. We fix a bounded open set $G$ with Lipschitz boundary such that $\Omega' \Subset G \Subset \Omega$. Let $\eta \in C^\infty_c(G)$ be a function such that $\eta =1$ on $\Omega'$, $0 \leq \eta \leq 1$, and $|\nabla \eta| \leq C$ for some $C>0$. By Lemma~\ref{lem-cap-zero}, there exist functions $\bar{\eta}_j \in W^{s, 1}(\mathbb{R}^n)$ such that $\bar{\eta}_j=0$ in a neighborhood of $K$, $0 \leq \bar{\eta}_j \leq 1$ in $\mathbb{R}^n$, $\lim_{j\to \infty} [\bar{\eta}_j]_{W^{s,1}(\mathbb{R}^n)} = 0$, and $\bar{\eta}_j \to 1$ a.e.\ in $\mathbb{R}^n$ as $j \to \infty$. Then, the estimate \eqref{eq-Caccio1-sub} with $\beta=1$ and $\bar{\eta}=\bar{\eta}_j$ shows that
\begin{align*}
I
&\coloneqq \int_{G} \int_{G} \frac{|\bar{u}(x) - \bar{u}(y)|}{|x-y|^{n+s}} (\eta\bar{\eta}_j)(y) \,\mathrm{d}y \,\mathrm{d}x \\
&\leq C_1 \int_{G} \int_{G} (\bar{u}(x)-k+\lambda) \frac{|(\eta\bar{\eta}_j)(x)-(\eta\bar{\eta}_j)(y)|}{|x-y|^{n+s}} \,\mathrm{d}y \,\mathrm{d}x \\
&\quad + C_1 \left( (\mathrm{diam}\, G)^{1-s} + \sup_{x \in \mathrm{supp}\,\eta} \int_{\mathbb{R}^n \setminus G} \frac{\mathrm{d}y}{|x-y|^{n+s}} \right) \int_{G} (\bar{u}-k+\lambda) \,\mathrm{d}x \\
&\quad + C_1 \|f\|_{L^{n/s}(G \cap \{u>k\})} \|(\bar{u}-k+\lambda)\eta\bar{\eta}_j\|_{L^{1^{\ast}_s}(G)},
\end{align*}
where $\bar{u}$ is defined as in \eqref{eq-truncation} and $C_1=C_1(n, s, \Lambda, \lambda)>0$.

By the fractional Sobolev inequality in Theorem~\ref{thm-FSI}, we have that
\begin{align*}
\|(\bar{u}-k+\lambda)\eta\bar{\eta}_j\|_{L^{1^{\ast}_s}(G)}
&\leq C_2 \int_{G} \int_{G} (\bar{u}(x)-k+\lambda) \frac{|(\eta\bar{\eta}_j)(x)-(\eta\bar{\eta}_j)(y)|}{|x-y|^{n+s}} \,\mathrm{d}y \,\mathrm{d}x \\
&\quad + C_2 I + C_2 \|(\bar{u}-k+\lambda)\eta\bar{\eta}_j\|_{L^1(G)}
\end{align*}
for some $C_2=C_2(n, s, G)>0$. If we take $k$ sufficiently large so that
\begin{equation}\label{eq-k}
C_1C_2 \|f\|_{L^{n/s}(G \cap \{u>k\})} \leq 1/2,
\end{equation}
then
\begin{equation*}
\frac{1}{2} I \leq C \int_{G} \int_{G} (\bar{u}(x)-k+\lambda) \frac{|(\eta\bar{\eta}_j)(x)-(\eta\bar{\eta}_j)(y)|}{|x-y|^{n+s}} \,\mathrm{d}y \,\mathrm{d}x + C \int_G (\bar{u}-k+\lambda) \,\mathrm{d}x
\end{equation*}
for some $C>0$, where we computed
\begin{equation*}
\sup_{x \in \mathrm{supp}\,\eta} \int_{\mathbb{R}^n \setminus G} \frac{\mathrm{d}y}{|x-y|^{n+s}} \leq \sup_{x \in \mathrm{supp}\,\eta} \int_{\mathbb{R}^n \setminus B_d(x)} \frac{\mathrm{d}y}{|x-y|^{n+s}} \leq \frac{|\mathbb{S}^{n-1}|}{s d^s}
\end{equation*}
with $d=\mathrm{dist}(\mathrm{supp}\,\eta, \partial G)>0$. Another application of the fractional Sobolev inequality and the triangle inequality
\begin{equation*}
|(\eta\bar{\eta}_j)(x)-(\eta\bar{\eta}_j)(y)| \leq |\bar{\eta}_j(x)-\bar{\eta}_j(y)| + |\eta(x)-\eta(y)|
\end{equation*}
yield that
\begin{align*}
&\|(\bar{u}-k+\lambda)\eta\bar{\eta}_j\|_{L^{1^{\ast}_s}(G)} \\
&\leq C \int_{G} \int_{G} (\bar{u}(x)-k+\lambda) \frac{|(\eta\bar{\eta}_j)(x)-(\eta\bar{\eta}_j)(y)|}{|x-y|^{n+s}} \,\mathrm{d}y \,\mathrm{d}x + C \int_{G} (\bar{u}-k+\lambda) \,\mathrm{d}x \\
&\leq C (l-k+\lambda) [\bar{\eta}_j]_{W^{s, 1}(G)} + C \int_{G} (\bar{u}-k+\lambda) \,\mathrm{d}x.
\end{align*}
Since $\lim_{j\to \infty} [\bar{\eta}_j]_{W^{s,1}(\mathbb{R}^n)} = 0$, by letting $j \to \infty$, we have that
\begin{align*}
\|\bar{u}-k\|_{L^{1^{\ast}_s}(\Omega')}
&\leq \|\bar{u}-k+\lambda\|_{L^{1^{\ast}_s}(\Omega')} \\
&\leq \|(\bar{u}-k+\lambda)\eta\|_{L^{1^{\ast}_s}(G)} \\
&\leq C_3 \|\bar{u}-k+\lambda\|_{L^1(G)} \\
&\leq C_3 \|\bar{u}-k\|_{L^1(\Omega')} + C_3 \|\bar{u}-k\|_{L^1(G\setminus \Omega')} + C_3 \lambda |G|
\end{align*}
for some $C_3>0$. Observing that $\bar{u}-k=0$ in $\{u \leq k\}$ and using H\"older's inequality, we obtain that
\begin{equation*}
\|\bar{u}-k\|_{L^1(\Omega')} = \|\bar{u}-k\|_{L^1(\Omega' \cap \{u > k\})} \leq |\Omega' \cap \{u>k\}|^{s/n} \|\bar{u}-k\|_{L^{1^{\ast}_s}(\Omega')}.
\end{equation*}
By taking $k$ sufficiently large so that \eqref{eq-k} and $C_3 |\Omega' \cap \{u>k\}|^{s/n} \leq 1/2$ hold, we have that
\begin{equation*}
\|\bar{u}-k\|_{L^{1^{\ast}_s}(\Omega')} \leq C \|\bar{u}-k\|_{L^1(G\setminus \Omega')} + C
\end{equation*}
for some $C>0$. Letting $l$ tend to infinity shows that
\begin{equation*}
\|u-k\|_{L^{1^{\ast}_s}(\Omega' \cap \{u>k\})} \leq C \|u-k\|_{L^1((G\setminus \Omega')\cap \{u>k\})} + C,
\end{equation*}
and consequently,
\begin{align}\label{eq-step2}
\begin{split}
\|u_+\|_{L^{1^{\ast}_s}(\Omega')}
&\leq C \|u\|_{L^{1^{\ast}_s}(\Omega' \cap \{u>k\})} + C \|u\|_{L^{1^{\ast}_s}(\Omega' \cap \{0<u\leq k\})} \\
&\leq C \|u-k\|_{L^{1^{\ast}_s}(\Omega' \cap \{u>k\})} + C \|k\|_{L^{1^{\ast}_s}(\Omega' \cap \{u>k\})} + C \|k\|_{L^{1^{\ast}_s}(\Omega')} \\
&\leq C \|u-k\|_{L^1((G\setminus \Omega')\cap \{u>k\})} + C + 2Ck|\Omega'|^{1-s/n}.
\end{split}
\end{align}
Since $u \in L^1_{\mathrm{loc}}(\Omega \setminus K)$ and $G \setminus \Omega' \Subset \Omega \setminus K$, the right-hand side of \eqref{eq-step2} is finite.

Similarly, one can show that $u_- \in L^{1^{\ast}_s}(\Omega')$ by using the estimate \eqref{eq-Caccio1-super} instead of \eqref{eq-Caccio1-sub}. This concludes that $u \in L^{1^{\ast}_s}(\Omega')$.
\end{proof}

We complete the proof of Theorem~\ref{thm-main} by applying the two previous lemmas.

\begin{proof}[Proof of Theorem~\ref{thm-main}]
By Lemmas~\ref{lem-step1} and \ref{lem-step2}, we have $u \in W^{s, 1}_{\mathrm{loc}}(\Omega)$. It only remains to prove that $u$ satisfies \eqref{eq-main-weak} for any $\varphi \in C^\infty_c(\Omega)$. As always, Lemma~\ref{lem-cap-zero} shows that there exists a sequence of functions $\bar{\eta}_j \in W^{s, 1}(\mathbb{R}^n)$ such that $\bar{\eta}_j=0$ in a neighborhood of $K$, $0 \leq \bar{\eta}_j \leq 1$ in $\mathbb{R}^n$, $\lim_{j\to \infty} [\bar{\eta}_j]_{W^{s, 1}(\mathbb{R}^n)} = 0$, and $\bar{\eta}_j \to 1$ a.e.\ in $\mathbb{R}^n$ as $j \to \infty$. We define $\varphi_j=\varphi \bar{\eta}_j$, then $\varphi_j \in W^{s, 1}(\Omega \setminus K) \cap L^\infty(\Omega \setminus K)$ and $\mathrm{supp}\,\varphi_j \Subset \Omega \setminus K$. It thus follows from Proposition~\ref{prop-test} that
\begin{equation}\label{eq-weak-j}
\int_{\mathbb{R}^n} \int_{\mathbb{R}^n} \mathscr{A} \left( x, y, \frac{u(x)-u(y)}{|x-y|} \right) (\varphi_j(x)-\varphi_j(y)) \frac{\mathrm{d}y \,\mathrm{d}x}{|x-y|^{n+s}} + \int_{\Omega \setminus K} \mathscr{B}(x, u) \varphi_j \,\mathrm{d}x = 0.
\end{equation}

On the one hand, by using the assumption \eqref{eq-A2} and
\begin{align*}
|(\varphi_j-\varphi)(x) - (\varphi_j-\varphi)(y)|
&=|\varphi(y)(\bar{\eta}_j(x)-\bar{\eta}_j(y))-(1-\bar{\eta}_j(x))(\varphi(x)-\varphi(y))| \\
&\leq |\varphi(y)| |\bar{\eta}_j(x)-\bar{\eta}_j(y)| + (1-\bar{\eta}_j(x)) |\varphi(x)-\varphi(y)|,
\end{align*}
we obtain that
\begin{align*}
&\left| \int_{\mathbb{R}^n} \int_{\mathbb{R}^n} \mathscr{A} \left( x, y, \frac{u(x)-u(y)}{|x-y|} \right) ((\varphi_j-\varphi)(x)-(\varphi_j-\varphi)(y)) \frac{\mathrm{d}y \,\mathrm{d}x}{|x-y|^{n+s}} \right| \\
&\leq \Lambda \|\varphi\|_{L^\infty(\mathbb{R}^n)} [\bar{\eta}_j]_{W^{s, 1}(\mathbb{R}^n)} + \Lambda \int_{\mathbb{R}^n} \int_{\mathbb{R}^n} (1-\bar{\eta}_j(x)) \frac{|\varphi(x)-\varphi(y)|}{|x-y|^{n+s}} \,\mathrm{d}y \,\mathrm{d}x \to 0
\end{align*}
as $j \to \infty$. Note that we used $\lim_{j\to \infty} [\bar{\eta}_j]_{W^{s, 1}(\mathbb{R}^n)} = 0$ for the first term and the dominated convergence theorem for the second term.

On the other hand, since $u \in L^1_{\mathrm{loc}}(\Omega)$, we have $\mathscr{B}(x, u) \in L^1_{\mathrm{loc}}(\Omega)$ by \eqref{eq-B1}. Moreover, the set $K$ has measure zero by Lemma~\ref{lem-measure}. Thus, we obtain that
\begin{align}\label{eq-step3-B}
\begin{split}
&\left| \int_{\Omega \setminus K} \mathscr{B}(x, u)\varphi_j \,\mathrm{d}x - \int_{\Omega} \mathscr{B}(x, u)\varphi \,\mathrm{d}x \right| \\
&= \left| \int_{\Omega} \mathscr{B}(x, u)(\varphi_j-\varphi) \,\mathrm{d}x \right| \leq \|\varphi\|_{L^\infty(\Omega)} \int_{\mathrm{supp}\,\varphi} |\mathscr{B}(x, u)| (1-\bar{\eta}_j) \,\mathrm{d}x \to 0
\end{split}
\end{align}
as $j \to \infty$ by the dominated convergence theorem. Therefore, we arrive at \eqref{eq-main-weak} by passing from \eqref{eq-weak-j} to the limit as $j \to \infty$.
\end{proof}

\section{Removable singularity theorem for weak solutions}\label{sec-weak}

In this section, we prove the removable singularity theorem, Theorem~\ref{thm-main-weak}, for weak solutions instead of solutions. We recall the definition of weak solutions for \eqref{eq-main} from Definition~\ref{def-weak}. For this purpose, we assume additionally that $\mathscr{B}$ is independent of $z$ variable and belongs to $L^1(\Omega)$, i.e.\
\begin{equation}\label{eq-B3}
\mathscr{B}(x, z)=\mathscr{B}(x) \in L^1_{\mathrm{loc}}(\Omega).
\end{equation}
The assumption \eqref{eq-B3} is stronger than \eqref{eq-B1} and \eqref{eq-B2}, but the equation \eqref{eq-main} under the assumptions \eqref{eq-A1}, \eqref{eq-A2}, \eqref{eq-A3}, and \eqref{eq-B3} still covers Example~\ref{examples}~(i) and (iii).

\begin{theorem}\label{thm-main-weak}
Let $\Omega \subset \mathbb{R}^n$ be open and let $E \subset \Omega$ be a relatively closed set of $(s, 1)$-capacity zero. Assume that $\mathscr{B}$ satisfies \eqref{eq-B3}. If $u$ is a weak solution of \eqref{eq-main} in $\Omega \setminus E$, then $u$ has a representative that is a weak solution of \eqref{eq-main} in all of $\Omega$.
\end{theorem}

The novelty of Theorem~\ref{thm-main-weak} is that the set $E$ can approach the boundary of $\Omega$. However, we do not know whether weak solutions are continuous, as the regularity theory is not available in this context. It is an interesting open question whether weak solutions are necessarily of class $W^{s, 1}_{\mathrm{loc}}$.

In fact, the proof of Theorem~\ref{thm-main-weak} for  the principal part (the left-hand side) of the equation \eqref{eq-main} is contained in the proof of Theorem~\ref{thm-main} because this part of the proof does not require any regularity on $u$ besides the measurability.

\begin{proof}[Proof of Theorem~\ref{thm-main-weak}]
The proof goes exactly the same way as in the proof of Theorem~\ref{thm-main}, except that we do not need $u \in W^{s, 1}_{\mathrm{loc}}(\Omega)$ and that we now have that
\begin{align*}
\left| \int_{\Omega \setminus E} \mathscr{B}\varphi_j \,\mathrm{d}x - \int_{\Omega} \mathscr{B} \varphi \,\mathrm{d}x \right|
&= \left| \int_{\Omega} \mathscr{B}(\varphi_j-\varphi) \,\mathrm{d}x \right| \\
&\leq \|\varphi\|_{L^\infty(\Omega)} \int_{\mathrm{supp}\,\varphi} |\mathscr{B}| (1-\bar{\eta}_j) \,\mathrm{d}x \to 0
\end{align*}
as $j \to \infty$, instead of \eqref{eq-step3-B}. Note that $\varphi_j=\varphi \bar{\eta}_j \in W^{s, 1}(\Omega \setminus E) \cap L^\infty(\Omega \setminus E)$ and it is compactly supported in $\Omega \setminus E$, as $\varphi$ is compactly supported in $\Omega$ and $\bar{\eta}_j$ vanishes in a neighborhood of $E$. Therefore, Proposition~\ref{prop-test} with $\Omega$ replaced by $\Omega \setminus E$ shows that \eqref{eq-weak-j} holds with $E$ in place of $K$. The rest of the argument remains valid for the principal part $\mathscr{A}$.
\end{proof}


\begin{thebibliography}{10}

\bibitem{AV14}
N.~Abatangelo and E.~Valdinoci.
\newblock A notion of nonlocal curvature.
\newblock {\em Numer. Funct. Anal. Optim.}, 35(7-9):793--815, 2014.

\bibitem{ADPM11}
L.~Ambrosio, G.~De~Philippis, and L.~Martinazzi.
\newblock Gamma-convergence of nonlocal perimeter functionals.
\newblock {\em Manuscripta Math.}, 134(3-4):377--403, 2011.

\bibitem{BFV14}
B.~Barrios, A.~Figalli, and E.~Valdinoci.
\newblock Bootstrap regularity for integro-differential operators and its
  application to nonlocal minimal surfaces.
\newblock {\em Ann. Sc. Norm. Super. Pisa Cl. Sci. (5)}, 13(3):609--639, 2014.

\bibitem{Ber51}
L.~Bers.
\newblock Isolated singularities of minimal surfaces.
\newblock {\em Ann. of Math. (2)}, 53:364--386, 1951.

\bibitem{BBK24}
A.~Bj{\"o}rn, J.~Bj{\"o}rn, and M.~Kim.
\newblock Perron solutions and boundary regularity for nonlocal nonlinear
  {D}irichlet problems.
\newblock {\em arXiv preprint arXiv:2406.05994}, 2024.

\bibitem{BLV19}
C.~Bucur, L.~Lombardini, and E.~Valdinoci.
\newblock Complete stickiness of nonlocal minimal surfaces for small values of
  the fractional parameter.
\newblock {\em Ann. Inst. H. Poincar\'e{} C Anal. Non Lin\'eaire},
  36(3):655--703, 2019.

\bibitem{CC19}
X.~Cabr\'e and M.~Cozzi.
\newblock A gradient estimate for nonlocal minimal graphs.
\newblock {\em Duke Math. J.}, 168(5):775--848, 2019.

\bibitem{CRS10}
L.~Caffarelli, J.-M. Roquejoffre, and O.~Savin.
\newblock Nonlocal minimal surfaces.
\newblock {\em Comm. Pure Appl. Math.}, 63(9):1111--1144, 2010.

\bibitem{CV11}
L.~Caffarelli and E.~Valdinoci.
\newblock Uniform estimates and limiting arguments for nonlocal minimal
  surfaces.
\newblock {\em Calc. Var. Partial Differential Equations}, 41(1-2):203--240,
  2011.

\bibitem{CV13}
L.~Caffarelli and E.~Valdinoci.
\newblock Regularity properties of nonlocal minimal surfaces via limiting
  arguments.
\newblock {\em Adv. Math.}, 248:843--871, 2013.

\bibitem{CL21}
M.~Cozzi and L.~Lombardini.
\newblock On nonlocal minimal graphs.
\newblock {\em Calc. Var. Partial Differential Equations}, 60(4):Paper No. 136,
  72, 2021.

\bibitem{DGS65}
E.~De~Giorgi and G.~Stampacchia.
\newblock Sulle singolarit\`a{} eliminabili delle ipersuperficie minimali.
\newblock {\em Atti Accad. Naz. Lincei Rend. Cl. Sci. Fis. Mat. Nat. (8)},
  38:352--357, 1965.

\bibitem{DNPV12}
E.~Di~Nezza, G.~Palatucci, and E.~Valdinoci.
\newblock Hitchhiker's guide to the fractional {S}obolev spaces.
\newblock {\em Bull. Sci. Math.}, 136(5):521--573, 2012.

\bibitem{DSV16}
S.~Dipierro, O.~Savin, and E.~Valdinoci.
\newblock Graph properties for nonlocal minimal surfaces.
\newblock {\em Calc. Var. Partial Differential Equations}, 55(4):Art. 86, 25,
  2016.

\bibitem{FV17}
A.~Figalli and E.~Valdinoci.
\newblock Regularity and {B}ernstein-type results for nonlocal minimal
  surfaces.
\newblock {\em J. Reine Angew. Math.}, 729:263--273, 2017.

\bibitem{KL24}
M.~Kim and S.-C. Lee.
\newblock Singularities of solutions of nonlocal nonlinear equations.
\newblock {\em arXiv preprint arXiv:2410.13292}, 2024.

\bibitem{Lau88}
C.~P. Lau.
\newblock Removable singularities of solutions to a class of quasilinear
  nonuniformly elliptic equations.
\newblock {\em J. Differential Equations}, 71(2):234--245, 1988.

\bibitem{Lom18}
L.~Lombardini.
\newblock Approximation of sets of finite fractional perimeter by smooth sets
  and comparison of local and global {$s$}-minimal surfaces.
\newblock {\em Interfaces Free Bound.}, 20(2):261--296, 2018.

\bibitem{MS02}
V.~Maz'ya and T.~Shaposhnikova.
\newblock On the {B}ourgain, {B}rezis, and {M}ironescu theorem concerning
  limiting embeddings of fractional {S}obolev spaces.
\newblock {\em J. Funct. Anal.}, 195(2):230--238, 2002.

\bibitem{Mir77}
M.~Miranda.
\newblock Sulle singolarit\`a{} eliminabili delle soluzioni dell'equazione
  delle superficie minime.
\newblock {\em Ann. Scuola Norm. Sup. Pisa Cl. Sci. (4)}, 4(1):129--132, 1977.

\bibitem{Nit65}
J.~C.~C. Nitsche.
\newblock {\"U}ber ein verallgemeinertes {D}irichletsches {P}roblem f\"ur die
  {M}inimalfl\"achengleichung und hebbare {U}nstetigkeiten ihrer {L}\"osungen.
\newblock {\em Math. Ann.}, 158:203--214, 1965.

\bibitem{Ser64}
J.~Serrin.
\newblock Local behavior of solutions of quasi-linear equations.
\newblock {\em Acta Math.}, 111:247--302, 1964.

\bibitem{Ser65}
J.~Serrin.
\newblock Removable singularities of solutions of elliptic equations. {II}.
\newblock {\em Arch. Rational Mech. Anal.}, 20:163--169, 1965.

\bibitem{Sim77}
L.~Simon.
\newblock On a theorem of de {G}iorgi and {S}tampacchia.
\newblock {\em Math. Z.}, 155(2):199--204, 1977.

\bibitem{VV81}
J.~L. V\'azquez and L.~V\'eron.
\newblock Removable singularities of some strongly nonlinear elliptic
  equations.
\newblock {\em Manuscripta Math.}, 33(2):129--144, 1980/81.

\end{thebibliography}
\end{document}